\documentclass[a4paper,12pt]{amsart}

\makeindex

\usepackage{amsfonts,graphics,amsmath,amsthm,amsfonts,amscd,amssymb,amsmath,latexsym,euscript, enumerate}
\usepackage{epsfig,url}
\usepackage{flafter}
\usepackage[all,cmtip,line]{xy}
\usepackage{array}
\usepackage[english]{babel}
\usepackage{overpic}
\usepackage{subfig}
\usepackage{multirow}

\usepackage{wrapfig}
\usepackage{enumitem}

\usepackage[usenames,dvipsnames,svgnames,table]{xcolor}
\usepackage{graphicx}
\usepackage[bookmarks=true,bookmarksnumbered=true,
 pdftitle={},
 pdfsubject={},
 pdfauthor={},
 pdfkeywords={TeX, dvipdfmx, hyperref, color,},
 colorlinks=true, linkcolor=Blue, citecolor=RoyalBlue]
 {hyperref}

\theoremstyle{definition}

\theoremstyle{theorem}

\theoremstyle{corollary}

\newtheorem{theorem}{Theorem}[section]
\newtheorem{lemma}[theorem]{Lemma}
\newtheorem{proposition}[theorem]{Proposition}

\theoremstyle{corollary}
\newtheorem{corollary}[theorem]{Corollary}

\theoremstyle{definition}
\newtheorem{definition}[theorem]{Definition}

\newtheorem{example}[theorem]{Example}

\newtheorem{remark}[theorem]{Remark}

\numberwithin{equation}{section}

\newcommand\R{\mathbb{R}}
\newcommand\Q{\mathbb{Q}}
\newcommand\Z{\mathbb{Z}}
\renewcommand\P{\mathbb{P}}

\newcommand\eps{\varepsilon}

\newcommand{\suchthat}{\;\ifnum\currentgrouptype=16 \middle\fi|\;}

\DeclareMathOperator{\ord}{ord}
\DeclareMathOperator{\mult}{mult}

\DeclareMathOperator{\Supp}{Supp}

\DeclareMathOperator{\vol}{vol}

\DeclareMathOperator{\mass}{mass}

\newcommand{\bm}{\mathbf B_-}  
\newcommand{\bp}{\mathbf B_+}  

\newcommand{\okbd}{\Delta}

\DeclareMathOperator{\cha}{char}

\begin{document}

\title[Seshadri constants and Okounkov bodies revisited]{Seshadri constants and Okounkov bodies revisited}

\author{Jinhyung Park}
\address{Department of Mathematics, Sogang University, Seoul, Korea}
\email{parkjh13@sogang.ac.kr}

\author{Jaesun Shin}
\address{Department of Mathematical Sciences, KAIST, Daejeon, Korea}
\email{jsshin1991@kaist.ac.kr}

\subjclass[2010]{14C20}
\date{\today}
\keywords{Seshadri constant, Okounkov body, big divisor, filtered graded linear series}

\begin{abstract}
In recent years, the interaction between the local positivity of divisors and Okounkov bodies has attracted considerable attention, and there have been attempts to find a satisfactory theory of positivity of divisors in terms of convex geometry of Okounkov bodies. Many interesting results in this direction have been established by Choi--Hyun--Park--Won \cite{CHPW} and K\"{u}ronya--Lozovanu \cite{KL1, KL2, KL3} separately. The first aim of this paper is to give uniform proofs of these results. Our approach provides not only a simple new outlook on the theory but also proofs for positive characteristic in the most important cases. Furthermore, we extend the theorems on Seshadri constants to graded linear series setting. Finally, we introduce the integrated volume function to investigate the relation between Seshadri constants and filtered Okounkov bodies introduced by Boucksom--Chen \cite{BC}.
\end{abstract}

\maketitle


\section{Introduction}

Throughout the paper, we work over an algebraically closed field $\Bbbk$ of arbitrary characteristic unless otherwise stated.
Let $f \colon Y \to X$ be a birational morphism between smooth projective varieties of dimension $n$, and $V_\bullet$ be a graded linear series associated to a divisor $D$ on $X$. Fix an \emph{admissible flag} on $Y$:
$$
Y_\bullet: Y=Y_0\supseteq Y_1\supseteq \cdots\supseteq Y_{n-1}\supseteq Y_n=\{ y\}
$$
where each $Y_i$ is an irreducible subvariety of codimension $i$ in $Y$ and is smooth at the point $y$.
With this data, we can associate a convex set in Euclidean space
$$
\okbd_{Y_\bullet}(f^*V_\bullet) \subseteq \R_{\geq 0}^n.
$$
When $V_\bullet$ is the complete graded linear series of $D$, we put $\okbd_{Y_\bullet}(f^*D)=\okbd_{Y_\bullet}(f^*V_\bullet)$.
Based on ideas of Okounkov \cite{O1, O2}, this construction was introduced in all its generality by Kaveh--Khovanskii \cite{KK} and Lazarsfeld--Musta\c{t}\u{a} \cite{LM}. We will focus mainly on two examples. The first one is when $f$ is the identity. In this case, we call $\okbd_{Y_\bullet}(V_\bullet)$ the \emph{Okounkov body} of $V_\bullet$. The other one is when $f$ is the blow-up $\pi \colon \widetilde{X} \to X$ of $X$ at a point $x$ and $Y_\bullet$ is an infinitesimal admissible flag, in which the last $n-1$ elements are linear subspaces of the exceptional divisor $E \simeq \P^{n-1}$ of $\pi$. In this case, we call $\okbd_{Y_\bullet}(\pi^*V_\bullet)$ the \emph{infinitesimal Okounkov body} of $V_\bullet$ over $x$. 

In recent years, a considerable amount of research has been devoted to the study of the connection between local positivity of divisors and Okounkov bodies. This direction of research was first tackled in the surface case in \cite{KL3}. For higher dimensions, \cite{KL1} deals with the infinitesimal setting, and the local picture is completed in \cite{CHPW} with partial results in \cite{KL2}. We further refer to \cite{CPW, CPW2, CPW3}, \cite{DKMS}, \cite{I},  \cite{R} for related results.

It has been clear by the works  \cite{CHPW}, \cite{KL1, KL2, KL3} that standard simplices arise naturally in  Okounkov bodies.
Let $\textbf{e}_1, \ldots, \textbf{e}_n$ be the standard basis vectors for $\R^n$, and $\mathbf{0}$ be the origin of $\R^n$. For $\xi \geq 0$, set
$$
\begin{array}{l}
\blacktriangle_{\xi}^n:= \text{closed convex hull $\big( \mathbf{0}, \xi \textbf{e}_1, \ldots, \xi \textbf{e}_n  \big)$}\\
\widetilde{\blacktriangle}_{\xi}^n :=\text{closed convex hull $\big( \mathbf{0}, \xi \textbf{e}_1, \xi (\textbf{e}_1+\textbf{e}_2), \ldots, \xi (\textbf{e}_1+\textbf{e}_n) \big)$}.
\end{array}
$$
We call $\blacktriangle_{\xi}^n$ (resp. $\widetilde{\blacktriangle}_{\widetilde{\xi}}^n$) the \emph{standard simplex} (resp. \emph{inverted standard simplex}) of size $\xi$.

In this paper, we prove the following ampleness criterion in terms of Okounkov bodies, which may be regarded as an analogue result of Seshadri's ampleness criterion (cf. \cite[Theorem 1.4.13]{L}).

\begin{theorem}\label{main1}
Let $X$ be a smooth projective variety of dimension $n$, and $D$ be a big $\R$-divisor on $X$. Then the following are equivalent:
\begin{enumerate}[wide, labelindent=0pt]
 \item[$(1)$] $D$ is ample.
 \item[$(2)$] For every point $x \in X$, there is an admissible flag $Y_\bullet$ centered at $x$ such that $\okbd_{Y_\bullet}(D)$ contains a nontrivial standard simplex in $\R_{\geq 0}^n$.
 \item[$(3)$] For every point $x \in X$, there is an infinitesimal admissible flag $\widetilde{Y}_\bullet$ over $x$ such that $\widetilde{\okbd}_{\widetilde{Y}_\bullet}(D)$ contains a nontrivial inverted standard simplex in $\R_{\geq 0}^n$.
\end{enumerate}
\end{theorem}

In characteristic zero, the equivalences $(1) \Longleftrightarrow (2)$ and $(1) \Longleftrightarrow (3)$ were proved in  \cite[Corollary D]{CHPW} and \cite[Theorem B]{KL1}, respectively. See Theorem \ref{thm-main1} for the more precise version.

Theorem \ref{main1} follows from the description of Seshadri constants in terms of Okounkov bodies. The \emph{Seshadri constant} $\eps(V_\bullet ;x)$ of a graded linear series $V_\bullet$ at a point $x$ is a measure of local positivity. It was first introduced by Demailly \cite{D}, and there has been a great deal of effort over the decades to study the Seshadri constants. 
See Section \ref{sec_prelim} for the precise definition.
As was shown in \cite{CHPW}, \cite{KL1, KL2, KL3}, the Seshadri constant $\eps(V_\bullet ;x)$ is closely related to the following constants
$$
\xi(V_\bullet ;x):=\sup_{Y_\bullet} \big\{  \xi \mid \blacktriangle_{\xi}^n \subseteq \okbd_{Y_\bullet}(V_\bullet ) \big\}  ~~ \text{and} ~~ \widetilde\xi (V_\bullet ;x):=\sup_{\widetilde{Y}_\bullet} \big\{ \widetilde{\xi} \mid \widetilde{\blacktriangle}_{\widetilde{\xi}}^{n} \subseteq \okbd_{\widetilde{Y}_\bullet}(\pi^*V_\bullet)  \big\},
$$
where $Y_\bullet$ runs over admissible flags centered at $x$ on $X$ and $\widetilde{Y}_\bullet$ runs over infinitesimal admissible flag over $x$. If no (resp. inverted) standard simplex is contained in the (resp. infinitesimal) Okounkov body, then we put $\xi_{Y_\bullet}(V_\bullet;x)=0$ (or $\widetilde{\xi}_{\widetilde{Y}_\bullet}(V_\bullet;x)=0$).

The following is the main result of the paper, which gives the description of Seshadri constants in terms of Okounkov bodies.

\begin{theorem}\label{main2}
Let $X$ be a smooth projective variety, $x \in X$ be a point, and $V_\bullet$ be a graded linear series associated to a divisor $D$ on $X$. Then we have
$$
\eps(V_\bullet; x) = \widetilde{\xi}(V_\bullet; x) \geq \xi (V_\bullet; x)
$$
in the following cases:
 \begin{enumerate}[wide]
 \item[$(1)$] $(\cha(\Bbbk)\geq 0)$ $V_\bullet$ is complete, and $D$ is nef and big.
 \item[$(2)$] $(\cha(\Bbbk)\geq 0)$ $V_\bullet$ is complete, and $x \not\in \bp(D)$.
 \item[$(3)$] $(\cha(\Bbbk) = 0)$ $V_\bullet$ is complete.
 \item[$(4)$] $(\cha(\Bbbk) = 0)$ $V_\bullet$ is birational, and $x \in X$ is very general.
 \end{enumerate}
\end{theorem}

Theorem \ref{main2} (3) was shown in \cite[Theorem E]{CHPW} for the inequality $\eps(V_\bullet; x)  \geq \xi (V_\bullet; x)$ and \cite[Theorem C]{KL1} for the equality $\eps(V_\bullet; x) = \widetilde{\xi}(V_\bullet; x)$. These works can be regarded as attempts to find a satisfactory theory of positivity of divisors in terms of convex geometry of Okounkov bodies. Another important result in this direction is the description of the augmented base locus $\bp(D)$ via Okounkov bodies proved in \cite[Theorem C]{CHPW} and \cite[Theorem 4.1]{KL1} (see also \cite{KL2, KL3}).

Even though the main theorems in \cite{CHPW} and  \cite{KL1} have the same nature and both depend on the deep results from \cite{ELMNP} about the continuity property for restricted volumes, the proofs look very different. The main technical ingredient of  \cite{KL1} is the interaction between infinitesimal Okounkov bodies of $D$ and jet separation of the adjoint divisor $K_X+D$ (see \cite[Proposition 4.10]{KL1}), but the main technical ingredients of \cite{CHPW} are the slice theorem of Okounkov bodies \cite[Theorem 1.1]{CPW} and a version of Fujita approximation \cite[Proposition 3.7]{BL}. We point out that the techniques aforementioned are based on Nadel vanishing theorem for multiplier ideal sheaves, so the characteristic zero assumption is necessary. 

In this paper, we give a new outlook on this theory by proving the main results of \cite{CHPW} and \cite{KL1} in a uniform way. Our proofs are shorter and simpler than those in \cite{CHPW} and \cite{KL1}.
After proving some basic lemmas, we first give quick direct proofs of Theorem \ref{main1} and Theorem \ref{main2} (1), (2). Our approach is elementary, avoiding the use of vanishing theorems. Consequently, these theorems hold in arbitrary characteristic.
For Theorem \ref{main2} (3), (4), we need to assume ${\rm char}(\Bbbk)=0$ because we apply the continuity result about moving Seshadri constants \cite[Theorem 6.2]{ELMNP} and  the differentiation result \cite[Proposition 2.3]{EKL}, \cite[Lemma 1.3]{N}. The augmented base locus results \cite[Theorem 4.1]{KL1} and \cite[Theorem C]{CHPW} then immediately follow (see Corollary \ref{theorem:0 equivalence}).

Note that the relation between Seshadri constants and Okounkov bodies for a birational graded linear series $V_\bullet$ was first studied by Ito \cite{I}. The inequality $\eps(V_\bullet; x) \geq \xi(V_\bullet;x)$ in Theorem \ref{main2} (4) may follow from \cite[Theorem 1.2]{I}, but our approach gives an alternative proof. The equality $\eps(V_\bullet; x) = \widetilde{\xi}(V_\bullet;x)$ is an original result. We remark that Theorem \ref{main2} (4) does not hold for a non-general point (see Remark  \ref{remark:gls}).

As was observed in \cite[Example 7.4]{CHPW}, \cite[Remark 4.9]{KL3}, the inequality $\eps(V_\bullet ;x) \geq \xi(V_\bullet ;x)$ in Theorem \ref{main2} can be strict in general. Moreover, one can conclude from \cite[Remark 3.12]{CPW3} that it is impossible to extract the exact value of $\eps(V_\bullet ;x)$ from the set of non-infinitesimal Okounkov bodies.
Thus it is necessary to consider finer structures on Okounkov bodies in order to read off the exact value of the Seshadri constant. For this purpose, we consider the multiplicative filtration $\mathcal{F}_x$ determined by the geometric valuation $\ord_x$ for a point $x \in X$ as
$$
\mathcal{F}_x^tV_m:=\{ s \in V_m \mid \ord_x(s) \geq t \}.
$$
This multiplicative filtration was treated in \cite{DKMS}, \cite{KMS}.
Now, fix an admissible flag $Y_\bullet$ on $Y$ centered at $x$.
With this data, we can associate a convex subset in Euclidean space
$$
\widehat{\Delta}_{Y_{\bullet}}(f^* V_\bullet, \mathcal{F}_x) \subseteq \R_{\geq 0}^{n+1},
$$
called the \emph{filtered Okounkov body}.
This was introduced in \cite{BC}.
We then define  the \emph{integrated volume function} as
$$
\widehat{\varphi}_x(V_\bullet, \mathcal{F}_x, t):=\int_{u=0}^{t} {\rm vol}_{\mathbb{R}^{n}}(\widehat{\Delta}_{Y_{\bullet}}(f^* V_{\bullet}, \mathcal{F}_x)_{x_{n+1}=u})du.
$$
We show that the derivative $\widehat{\varphi}_{x}'(V_\bullet, \mathcal{F}_x, t)$ always exists (see Proposition \ref{prop:basicfilvol}).
Note that
$$
\vol_{\R^{n+1}} (\widehat{\Delta}_{Y_{\bullet}}(f^* V_\bullet, \mathcal{F}_x)) = \widehat{\varphi}_x(V_\bullet, \mathcal{F}_x, \infty).
$$
The value $\widehat{\varphi}_x(V_\bullet, \mathcal{F}_x, \infty)$ has been used to study diophantine approximation on algebraic varieties in \cite{MR} and the K-stability of Fano varieties (cf. \cite{BJ}, \cite{F}, \cite{Li}).

The following theorem gives a new characterization of the Seshadri constant in terms of the integrated volume function.

\begin{theorem}\label{main3}
Let $X$ be a smooth projective variety of dimension $n$. Let $x \in X$ be a point, and $V_\bullet$ be a graded linear series associated to a divisor $D$ on $X$. Then we have
$$
\eps(V_\bullet;x)=\inf \left \{ t \geq 0 \suchthat \widehat{\varphi}_{x}'(V_\bullet, \mathcal{F}_x, 0)-\widehat{\varphi}_{x}'(V_\bullet, \mathcal{F}_x, t)<\frac{t^n}{n!} \right \}
$$
in the four cases considered in Theorem \ref{main2}.
\end{theorem}

In Section \ref{sec_filokbd}, we define the \emph{bounded mass function} $\mass_+(V_m, \mathcal{F}_x, t)$ for $t \geq 0$ as an ``appropriate'' sum of jumping numbers of $(V_m, \mathcal{F}_x)$, and we show in Theorem \ref{theorem:mass} that
$$
\widehat{\varphi}_{x}(V_\bullet, \mathcal{F}_x, t)=\lim_{m \rightarrow \infty} \frac{\mass_+(V_{m}, \mathcal{F}_x, mt)}{m^{n+1}}.
$$
This means that the integrated volume function is independent of the choice of the admissible flags to define the filtered Okounkov body.

\medskip

The rest of the paper is organized as follows. We begin in Section \ref{sec_prelim} with recalling basic definitions. In Section \ref{sec_okbd}, we first show some basic lemmas, and then, give proofs of Theorems \ref{main1} and \ref{main2}. Section \ref{sec_filokbd} is devoted to the study of integrated volume functions; in particular, we prove Theorem \ref{main3}.

\subsection*{Acknowledgement}
The authors would like to thank the referee for careful reading of the paper and useful suggestions to help improve the exposition of the paper. J. Park was partially supported by the Sogang University Research Grant of 201910002.01.

\section{Preliminaries}\label{sec_prelim}

\subsection{Notations}\label{subsec-nota}
Throughout the paper, we fix the following notations.
Let $X$ be a smooth projective variety of dimension $n$ defined over an algebraically closed field $\Bbbk$ of arbitrary characteristic, and $D$ be an $\R$-divisor on $X$. Let $x \in X$ be a point, and $\pi \colon \widetilde{X} \to X$ be the blow-up of $X$ at $x$ with the exceptional divisor $E$. Let $V_\bullet$ be a graded linear series associated to $D$ so that $V_m$ is a linear subspace of $H^0(X, \mathcal{O}_X(\lfloor mD \rfloor) )$ for every integer $m \geq 0$. Recall the following definitions:
 \begin{enumerate}
 \item[$(1)$] $V_\bullet$ is called \emph{complete} if $V_m=H^0(X, \mathcal{O}_X(\lfloor mD \rfloor) )$ for all $m \geq 0$.
 \item[$(2)$] $V_\bullet$ is called \emph{birational} if the rational map given by $|V_m|$ is birational onto its image for any $m \gg 0$. It is exactly same to Condition (B) in \cite[Definition 2.5]{LM}.
 \end{enumerate}

For each integer $m \geq 1$, let $f_m \colon X_m \to X$ be a birational morphism such that $X_m$ is a normal projective variety and $f_m^{-1}\mathfrak{b}(V_m) \cdot \mathcal{O}_{X_m} = \mathcal{O}_{X_m}(-F_m)$ for an effective Cartier divisor $F_m$ on $X_m$. In characteristic zero, we may assume that $f_m$ is a log resolution of the base ideal $\mathfrak{b}(V_m)$.
In positive characteristic, instead of using the resolution of singularities, we construct $f_m$ by taking the normalization of the blow-up along $\mathfrak{b}(V_m)$.
We have a decomposition
$$
f_m^*|V_m|=|W_m|+F_m,
$$
where $W_m \subseteq H^0(X_m, \mathcal{O}_{X_m}(f_m^*\lfloor mD \rfloor) - F_m)$ is a linear subspace defining a free linear series. We set
$$
M_m:=f_m^*\lfloor mD \rfloor - F_m~~\text{ and }~~M_m':=\frac{1}{m}M_m, ~ F_m':=\frac{1}{m}F_m.
$$
Suppose now that $f_m$ is isomorphic over a neighborhood of $x$ and $f_m^{-1}(x) \not\subseteq \Supp(F_m)$. Let $\pi_m \colon \widetilde{X}_m \to X_m$ be the blow-up at the smooth point $f_m^{-1}(x)$ with the exceptional divisor $E_m$. 
We have the following commutative diagram
\begin{displaymath}
\xymatrix{
\widetilde{X}_{m} \ar[d]_{\widetilde{f}_m} \ar[r]^{\pi_m} & X_{m}  \ar[d]^{f_{m}} \\
\widetilde{X} \ar[r]_{\pi} & X.
}
\end{displaymath}

\subsection{Okounkov bodies}
Let $Y$ be a projective variety, and fix an admissible flag on $Y$:
$$
Y_\bullet: Y=Y_0\supseteq Y_1\supseteq \cdots\supseteq Y_{n-1}\supseteq Y_n=\{ y\}.
$$
Here every $Y_i$ is smooth at the point $y$ for $0 \leq i \leq n$.
We define a valuation-like function
$$
\nu_{Y_\bullet} \colon V_\bullet \longrightarrow \R^{n}_{\geq 0}
$$
as follows: Given a nonzero section $s \in V_m$, let 
$$
\nu_1=\nu_1(s)=\ord_{Y_1}(s).
$$ 
After choosing a local equation for $Y_1$ in $Y$, we get $\widetilde{s}_1 \in H^0(X, \mathcal{O}_X(\lfloor mD \rfloor - \nu_1 Y_1))$. Let
$$
\nu_2=\nu_2(s)=\ord_{Y_2}(\widetilde{s_1}|_{Y_1}).
$$
Continuing the process, we obtain
$$
\nu_{Y_\bullet} (s) :=(\nu_1(s), \ldots, \nu_n(s)) \in \Z_{\geq 0}^n.
$$
Let $\Gamma(V_\bullet)_m \subseteq \Z_{\geq 0}^n$ be the image of $\nu_{Y_\bullet} \colon (V_m \setminus \{ 0 \} ) \to \Z^n_{\geq 0}$.
The \emph{Okounkov body} of $V_\bullet$ with respect to $Y_\bullet$ is defined as
$$
 \okbd_{Y_\bullet}(V_\bullet):=\text{closed convex hull } \left( \bigcup_{m \geq 1} \frac{1}{m} \Gamma(V_\bullet)_m \right) \subseteq \R^n_{\geq 0}.
$$
For more details and basic properties, we refer to \cite{KK}, \cite{LM}.

In this paper, we consider mainly two cases. The first one is when $Y=X$ and $Y_\bullet$ is centered at $x$. The other one is when $Y=\widetilde{X}$ is the blow-up of $X$ at a point $x$ and $Y_\bullet=\widetilde{Y}_\bullet$ is an infinitesimal admissible flag, in which the last $n-1$ elements are linear subspaces of $E \simeq \P^{n-1}$. In this case, we say that
$$
\widetilde{\okbd}_{\widetilde{Y}_\bullet}(V_\bullet):=\okbd_{\widetilde{Y}_\bullet}(\pi^*V_\bullet) \subseteq \R^n_{\geq 0}.
$$
is the \emph{infinitesimal Okounkov body} of $V_\bullet$ with respect to $\widetilde{Y}_\bullet$. 
When $V_\bullet$ is complete, we put $\okbd_{Y_\bullet}(D)=\okbd_{Y_\bullet}(D)$ and $\widetilde{\okbd}_{\widetilde{Y}_\bullet}(D)=\widetilde{\okbd}_{\widetilde{Y}_\bullet}(V_\bullet)$.
The infinitesimal Okounkov body was first introduced in \cite{LM} when $x$ is a very general point, and a similar construction is also considered in \cite{WN}. It was generalized to arbitrary point in \cite{KL1, KL2, KL3}. More general infinitesimal admissible flags for Okounkov bodies have been studied in \cite{R}, \cite{CPW3}.

Now, we define nonnegative numbers
$$
\begin{array}{rcl}
\xi_{Y_\bullet}(V_\bullet;x):=\max \{ \xi \mid \blacktriangle_{\xi}^n \subseteq \okbd_{Y_\bullet}(V_\bullet) \} & \text{and} & \widetilde{\xi}_{\widetilde{Y}_\bullet}(V_\bullet;x):=\max \{ \widetilde{\xi} \mid \widetilde{\blacktriangle}_{\widetilde{\xi}}^{n} \subseteq \widetilde{\okbd}_{\widetilde{Y}_\bullet}(V_\bullet) \},\\
\displaystyle  \xi(V_\bullet;x):=\sup_{Y_\bullet} \{ \xi_{Y_\bullet}(V_\bullet;x) \}  & \text{and} & \displaystyle  \widetilde\xi (V_\bullet;x):=\sup_{\widetilde{Y}_\bullet} \{ \widetilde{\xi}_{\widetilde{Y}_\bullet}(V_\bullet;x)  \},
\end{array}
$$
where the supremums run over all admissible flags $Y_\bullet$ centered at $x$ and all infinitesimal admissible flags over $x$, respectively. 
If no (resp. inverted) standard simplex is contained in the (resp. infinitesimal) Okounkov body, then we let $\xi_{Y_\bullet}(V_\bullet;x)=0$ (or $\widetilde{\xi}_{\widetilde{Y}_\bullet}(V_\bullet;x)=0$). 
When $V_\bullet$ is complete, we simply replace $V_\bullet$ by $D$.

In this paper, $\mathcal{F}_x$ is always the multiplicative filtration on $V_\bullet$ given by 
$$
\mathcal{F}_x^t V_m:=\{ s \in V_m \mid \ord_x(s) \geq t \}.
$$
Then $\mathcal{F}_x$ is pointwise bounded below and linearly bounded above in the sense of \cite[Definition 1.3]{BC} (see \cite[Proposition 3.5]{KMS}).
For any $t \in \R$, we have a new graded linear series $V_\bullet^{(t)}$, which is defined as
$V_m^{(t)}:=\mathcal{F}^{tm} V_m $ for all $m \in \Z_{\geq 0}$.
Notice that the Okounkov bodies $\okbd_{Y_\bullet}(V_\bullet^{(t)})$ form a nonincreasing family of convex subsets of $\okbd_{Y_\bullet}(V_\bullet)$. See \cite{BC} for more details.

\begin{lemma}\label{lem-multifilt}
If $V_\bullet$ is birational, then so is $V_\bullet^{(t)}$ for any $t > 0$ such that $V_m^{(t + \epsilon)} \neq \emptyset$ for any integer $m \gg 0$ and a sufficiently small number $\epsilon > 0$. 
\end{lemma}

\begin{proof}
Let $m' \gg 0$ be an integer such that $V_{m'}$ defines a birational map, and $\epsilon > 0$ be a sufficiently small number such that $V_{m}^{(t + \epsilon)} \neq \emptyset$ for $m \gg 0$.
 For any integer $m \gg 0$ with $m(t+\epsilon) > (m+m')t$, we have $s \cdot V_{m'} \subseteq V_{m+m'}^{(t)}$ for any zero section $s \in V_{m}^{(t + \epsilon)}$. Then $V_{m+m'}^{(t)}$ defines a birational map.
\end{proof}

\begin{example}\label{ex-multifiltinfokbd}
For an admissible flag  $Y_\bullet$ on $X$ centered at $x$, we have
\begin{equation}\label{eq-multifiltokbd1}
\okbd_{Y_\bullet}(V_\bullet^{(t)}) \subseteq \okbd_{Y_\bullet}(V_\bullet) \setminus \blacktriangle_{t}^n~~\text{ for any $t \geq 0$}.
\end{equation}
For an infinitesimal admissible flag $\widetilde{Y}_\bullet$ on $\widetilde{X}$ over $x$, we have
\begin{equation}\label{eq-multifiltokbd}
\widetilde{\okbd}_{\widetilde{Y}_\bullet}(V_\bullet^{(t)}) = \widetilde{\okbd}_{\widetilde{Y}_\bullet}(V_\bullet)_{x_1 \geq t}~~\text{ for any $t \geq 0$}.
\end{equation}
Now, assume that $V_\bullet$ is the complete graded linear series of $D$. For any $t \geq 0$, let $W_\bullet^t$ be the graded linear series on $\widetilde{Y}_1=E$ such that $W_m^t$ for any $m \geq 0$ is given by the image of the map
 $$
H^0(\widetilde{X}, \mathcal{O}_{\widetilde{X}}(\lfloor m(\pi^*D-tE) \rfloor)) \longrightarrow H^0(\widetilde{Y}_1, \mathcal{O}_{\widetilde{Y}_1} (\lfloor m(\pi^*D-tE) \rfloor)).
$$ 
Then we will show in Lemma \ref{lemma:upper bound} (2) that
$$
\widetilde{\okbd}_{\widetilde{Y}_\bullet}(V_\bullet)_{x_1=t} = \okbd_{\widetilde{Y}_\bullet | \widetilde{Y}_1 }(W_\bullet^t).
$$
\end{example}

Following \cite{BC}, we define the \emph{concave transform} of a multiplicative filtration $\mathcal{F}$ on $V_\bullet$ to be a real-valued function on $\okbd_{Y_\bullet}(V_\bullet)$ given by
$$
\varphi_{\mathcal{F}, Y_\bullet}(\mathbf{x}):=\sup \{ t \in \R \mid \mathbf{x} \in \okbd_{Y_\bullet}(V_\bullet^{(t)}) \},
$$
and the \emph{filtered Okounkov body} associated to $V_\bullet, \mathcal{F}$ with respect to $Y_\bullet$ to be  a compact convex subset of $\R_{\geq 0}^n \times \R_{\geq 0} = \R^{n+1}_{\geq 0}$ given by
$$
\widehat{\Delta}_{Y_{\bullet}}(V_{\bullet}, \mathcal{F}):=\{(\mathbf{x},t) \in \Delta_{Y_{\bullet}}(V_{\bullet}) \times \mathbb{R} \text{ $|$ } t \in [0, \varphi_{\mathcal{F_{\bullet}}, Y_{\bullet}}(\mathbf{x})]\} \subseteq \R^{n+1}_{\geq 0}.
$$
Note that
$$
\widehat{\Delta}_{Y_{\bullet}}(V_{\bullet}, \mathcal{F}_x)_{x_{n+1}=t} = \okbd_{Y_\bullet}(V_\bullet^{(t)}).
$$

\subsection{Seshadri constants}\label{subsec:ses}

Let $D$ be an arbitrary $\R$-divisor. 
The \emph{stable base locus} of $D$ is defined as 
$$
\text{SB}(D):=\bigcap_{D \sim_{\R} D' \geq 0} \Supp(D').
$$
Recall that $D \sim_{\R} D'$ if $D-D'$ is an $\R$-linear sum of principal divisors.
If there is no effective divisor $D'$ with $D' \sim_{\R} D$, then $\text{SB}(D)=X$.
The \emph{restricted base locus} of $D$ and the \emph{augmented base locus} of $D$ are defined as
$$
\bm(D):=\bigcup_{A:\text{ample}} \text{SB}(D+A)~~\text{ and }~~\bp(D):=\bigcap_ {A:\text{ample}} \text{SB}(D-A).
$$
Note that $D$ is nef if and only if $\bm(D)=\emptyset$, and $D$ is ample if and only if $\bp(D)=\emptyset$. 
Furthermore, $D$ is not big if and only if $\bp(D)=X$.
See \cite{ELMNP}, \cite{M} for further properties.

Now, for a given graded linear series $V_\bullet$ and a point $x \in X$, let $s(V_m; x)$ be the supremum of integers $s \geq -1$ such that the natural map 
$$
V_m \longrightarrow H^0(\mathcal{O}_X(\lfloor mD \rfloor) \otimes \mathcal{O}_X/\mathfrak{m}_x^{s+1})
$$
is surjective. The \emph{Seshadri constant} of $V_\bullet$ at $x$ is defined to be 
$$
\eps(V_\bullet;x):=\limsup_{m \to \infty} \frac{s(V_m;x)}{m}.
$$ 
If $V_\bullet$ is a complete graded linear series of $D$ and $D$ is nef, then $\eps(V_\bullet;x)$ coincides with the usual Seshadri constant
$$
\eps(D;x):=\sup\{ k \mid \pi^*D-kE \text{ is nef} \} = \inf_{x \in C} \left\{ \frac{D.C}{\mult_x C} \right\},
$$
where the infimum runs over all irreducible curves $C$ on $X$ passing through $x$.
Next, we define the \emph{moving Seshadri constant} $\eps(||D||;x)$ of a divisor $D$ at a point $x$ as follows:
$$
\begin{array}{l}
\text{If $x \in \bp(D)$, then $\eps(||D||;x):=0$.}\\
\text{If $x \not\in \bp(D)$, then $\displaystyle\eps(||D||;x):= \limsup_{m \to \infty} \eps(M_m'; f_m^{-1}(x))$.}
\end{array}
$$
By \cite[Theorem 6.2 and Propositions 6.4, 6.6]{ELMNP2} and \cite[Propositions 7.1.2, 7.2.3, 7.2.10]{Mur}, we have
$$
\eps(||D||;x)=\eps(V_\bullet;x)~~\text{when $V_\bullet$ is a complete graded linear series of $D$.}
$$
Note that $x \in \bp(D)$ if and only if $\eps(||D||;x)=0$.

\begin{lemma}\label{lem:seshample}
Suppose that $D$ is nef and big. Then $D$ is ample if and only if $\eps(D;x) > 0$ for all $x \in X$.
\end{lemma}

\begin{proof}
If $D$ is ample, then clearly $\eps(D;x) > 0$ for all $x \in X$. Assume that $\eps(D;x) > 0$ for all $x \in X$. Then $x \not\in \bp(D)$ for all $x \in X$, so $\bp(D)=\emptyset$. Thus $D$ is ample.
\end{proof}

\noindent 

When $\cha(\Bbbk) = 0$, the function 
$$
\eps(||-||;x) \colon N^1(X)_{\R} \longrightarrow \R_{\geq 0}
$$ 
is continuous (see \cite[Proposition 6.3]{ELMNP2}). When $\cha(\Bbbk) > 0$, the function
$$
\eps(||-||;x) \colon \text{Big}_{\R}^{\{ x \}}(X) \to \R_{>0}
$$
is continuous, where $\text{Big}_{\R}^{\{ x \}}(X) $ denotes the open convex subcone of the big cone consisting of big divisors classes $D$ such that $x \not\in \bp(D)$ (see \cite[Proposition 7.1.2]{Mur}).
For further details on Seshadri constants, we refer to \cite[Section 6]{ELMNP2}, \cite{I}, \cite[Chapter 5]{L}, and \cite[Section 7]{Mur}.

The \emph{Nakayama constant} of $V_\bullet$ at $x$ is defined by
$$
\mu(V_\bullet ;x)=\sup \left\{ \frac{\ord_x(s)}{m} \mid s \in V_m \right\}.
$$ 
When $V_\bullet$ is complete, we put $\mu(D;x) =\mu(V_\bullet;x)$. If $D$ is pseudoeffective, then
$$
\mu(D;x)=\sup \{ k \mid \pi^*D - kE~\text{ is pseudoeffective} \}.
$$

\section{Local positivity via Okounkov bodies}\label{sec_okbd}

In this section, we prove Theorems \ref{main1} and  \ref{main2}. 

\subsection{Basic Lemmas}
First, we show some useful lemmas.

\begin{lemma}\label{lemma:upper bound}
Suppose that $D$ is big and $x \not\in \bm(D)$. For $0<k<\mu(D;x)$, we have the following:
\begin{enumerate}[wide, labelindent=0pt]
 \item[$(1)$] $E \nsubseteq \bp(\pi^{*}D-kE)$.
 \item[$(2)$] $\widetilde{\okbd}_{\widetilde{Y}_\bullet}(D)_{x_1=k} = \okbd_{\widetilde{Y}_{\bullet}|\widetilde{Y}_1}(W_\bullet^k)$, where $W_\bullet^k$ is  defined in Example \ref{ex-multifiltinfokbd}.
 \end{enumerate}
\end{lemma} 
 
 \begin{proof}
Note that $E \not\subseteq \bm(\pi^*D)$. By \cite[Theorem A]{FKL}, we have 
$$
\vol_{\widetilde{X}}(\pi^*D-kE) \neq \vol_{\widetilde{X}}(\pi^*D).
$$
Then $(1)$ follows from \cite[Theorem B]{FKL}, and $(2)$ follows from \cite[Theorem 4.26]{LM}.
 \end{proof}

 \begin{lemma}\label{lemma:upper bound2}
 Suppose that $V_\bullet$ is birational. Then we have the following:
 \begin{enumerate}[wide, labelindent=0pt]
 \item[$(1)$] $\widetilde{\Delta}_{\widetilde{Y}_{\bullet}}(V_\bullet) \subseteq \widetilde{\blacktriangle}^n_{\mu(V_\bullet;x)}$ for any infinitesimal admissible flag $\widetilde{Y}_{\bullet}$ over $x$.
 \item[$(2)$] If $\widetilde{\blacktriangle}^n_{\widetilde{\xi}} \subseteq \widetilde{\Delta}_{\widetilde{Y}_{\bullet}}(V_\bullet)$ for some infinitesimal admissible flag $\widetilde{Y}_{\bullet}$ over $x$, then the same is true for every infinitesimal admissible flag over $x$. In particular, 
 $$
 \widetilde{\xi}(V_\bullet;x) = \widetilde{\xi}_{Y_\bullet'}(V_\bullet;x)
 $$ 
for every infinitesimal admissible flag $Y_{\bullet}'$ over $x$.
 \end{enumerate}
 \end{lemma}

\begin{proof}
This was proved in \cite[Propositions 2.5 and 4.7]{KL1} (their proof work for any graded linear series), but we give a proof for reader's convenience. Note that 
$$
\widetilde{\Delta}_{\widetilde{Y}_{\bullet}}(V_\bullet) \subseteq \widetilde{\Delta}_{\widetilde{Y}_{\bullet}}(D)_{x_1 \leq \mu(V_\bullet;x)}.
$$ 
Consider a graded linear series $W_\bullet^k$ on $\widetilde{Y}_1$ in Lemma \ref{lemma:upper bound} (2), which is the restriction of a complete graded linear series of $D$.
We have
$$
\okbd_{\widetilde{Y}_\bullet|\widetilde{Y}_1}(W_\bullet^k) \subseteq \okbd_{\widetilde{Y}_\bullet|\widetilde{Y}_1}(kH) = \blacktriangle^{n-1}_k,
$$
where $H$ is a hyperplane section of $\widetilde{Y}_1 \simeq \P^{n-1}$. This implies  (1).

For $(2)$, assume that $\widetilde{\blacktriangle}^n_{\widetilde{\xi}} \subseteq \widetilde{\Delta}_{\widetilde{Y}_{\bullet}}(V_\bullet)$ for some infinitesimal admissible flag $\widetilde{Y}_\bullet$ over $x$.
By $(1)$ and (\ref{eq-multifiltokbd}) in Example \ref{ex-multifiltinfokbd}, we have
$$
\widetilde{\okbd}_{\widetilde{Y}_\bullet}(V_\bullet) = \widetilde{\blacktriangle}^n_{k} \cup \widetilde{\okbd}_{\widetilde{Y}_\bullet}(V_\bullet^{(k)}).
$$
By \cite[Theorem 2.13]{LM}, we get
$$
\frac{1}{n!} \left( \vol_X(V_\bullet) - \vol_X(V_\bullet^{(k)})  \right)= \vol_{\R^n} \big( \widetilde{\blacktriangle}^n_{k} \big).
$$
Then $\vol_{\R^n}\big( \widetilde{\okbd}_{\widetilde{Y}_\bullet'}(V_\bullet)_{x_1 \leq k} \big) =  \vol_{\R^n} \big( \widetilde{\blacktriangle}^n_{k} \big)$ for any infinitesimal admissible flag $\widetilde{Y}_\bullet'$ over $x$. By $(1)$, we obtain $\widetilde{\okbd}_{\widetilde{Y}_\bullet'}(V_\bullet)_{x_1 \leq k} = \widetilde{\blacktriangle}^n_{k}$, which proves $(2)$.
\end{proof}

\begin{lemma}\label{lemma:origin}
Suppose that $D$ is big. If $\mathbf{0} \in \Delta_{Y_{\bullet}}(D)$ for some admissible flag $Y_{\bullet}$ on $X$ centered at $x$, then $x \notin \bm(D)$.
\end{lemma}

\begin{proof}
This was shown in \cite[Theorem A]{CHPW} or \cite[Theorem A]{KL2} (their proofs work for positive characteristic), but we include the proof for reader's convenience.
Suppose that $x \in \bm(D)$. By \cite[Proposition 2.8]{ELMNP} and \cite[Theorem C]{M}, $\ord_x(||D||) > 0$. Let $Y_\bullet$ be an admissible flag on $X$ centered at $x$. For any $(\nu_1, \ldots, \nu_n) \in \okbd_{Y_\bullet}(D)$, we have
\begin{equation}\label{eq-okbd_mult}
\nu_1 + \cdots + \nu_n \geq \ord_x(||D||).
\end{equation}
Then $\nu_1 + \cdots + \nu_n >0$, so $\mathbf{0} \not\in \okbd_{Y_\bullet}(D)$.
\end{proof}

\begin{lemma}\label{lemma:curve}
Suppose that $D$ is nef and big. For any $k$ with $\epsilon(D;x) < k <\mu(D;x)$, there is an irreducible curve $C$ on $X$ passing through $x$ such that $\overline{C} \subseteq \bm(\pi^*D-kE)$, where $\overline{C}$ is the strict transform of $C$ by $\pi$. In particular, $\mathbf{0} \not\in \okbd_{\widetilde{Y}_{\bullet}}(\pi^*D-kE)$ for any infinitesimal admissible flag $\widetilde{Y}_\bullet$ over $x$ centered at $x' \in \overline{C} \cap E$.
\end{lemma}

\begin{proof}
Note that $\pi^*D-kE$ is not nef. Thus there is an irreducible curve $\overline{C}$ on $\widetilde{X}$ such that $(\pi^*D-kE).\overline{C} < 0$, and consequently, $\overline{C}  \subseteq \bm(\pi^*D-kE)$. We know that $E.\overline{C} > 0$. Then $\overline{C} \cap E \neq \emptyset$ and $\overline{C} \not\subseteq \Supp(E)$. 
By letting $C:=\pi(\overline{C})$, we are done. 
Now, the `in particular' part follows from Lemma \ref{lemma:origin} because $x' \in \bm(\pi^*D-kE)$.
\end{proof}

The following two lemmas are the main observations of this paper.

\begin{lemma}\label{lemma:compare}
Suppose that $V_\bullet$ is birational. Then we have
$$
\widetilde{\xi}(V_\bullet;x) \geq \xi(V_\bullet; x).
$$
\end{lemma}

\begin{proof}
It is sufficient to show that
$$
\widetilde{\xi}:= \widetilde{\xi}_{\widetilde{Y}_\bullet}(V_\bullet;x) \geq \xi_{Y_\bullet}(V_\bullet ;x)=:\xi
$$
for an infinitesimal admissible flag $\widetilde{Y}_\bullet$ over $x$ and an admissible flag $Y_\bullet$ on $X$ centered at $x$. Fix a sufficiently small number $\epsilon > 0$. Then $\blacktriangle_{\xi-\epsilon}^n \subseteq \okbd_{Y_\bullet}(V_\bullet)$.
By (\ref{eq-multifiltokbd1}) in Example  \ref{ex-multifiltinfokbd} and \cite[Theorem 2.13]{LM}, we have
\begin{equation}\label{eq-compare1}
\vol_X(V_\bullet^{(\xi-\epsilon)}) \leq \vol_X(V_\bullet) -(\xi-\epsilon)^n.
\end{equation}
On the other hand, by (\ref{eq-multifiltokbd}) in Example  \ref{ex-multifiltinfokbd} and Lemma \ref{lemma:upper bound2} $(1)$, 
$$
\widetilde{\okbd}_{\widetilde{Y}_\bullet}(V_\bullet) \subseteq \widetilde{\blacktriangle}_{\xi-\epsilon}^n \cup \widetilde{\okbd}_{\widetilde{Y}_\bullet}(V_\bullet^{(\xi-\epsilon)}),
$$
so \cite[Theorem 2.13]{LM} implies that 
\begin{equation}\label{eq-compare2}
 \vol_X(V_\bullet) \leq (\xi-\epsilon)^n +\vol_X(V_\bullet^{(\xi-\epsilon)}).
\end{equation}
By comparing (\ref{eq-compare1}) and (\ref{eq-compare2}), we see that the equality holds, and hence, we obtain
$$
\widetilde{\okbd}_{\widetilde{Y}_\bullet}(V_\bullet) = \widetilde{\blacktriangle}_{\xi-\epsilon}^n \cup \widetilde{\okbd}_{\widetilde{Y}_\bullet}(V_\bullet^{(\xi-\epsilon)}).
$$ 
Since $\epsilon>0$ can be arbitrarily small, we get
$\widetilde{\blacktriangle}_{\xi}^n \subseteq \widetilde{\okbd}_{\widetilde{Y}_\bullet}(V_\bullet)$. 
This implies that $\widetilde{\xi} \geq \xi$.
\end{proof}

\begin{lemma}\label{lemma:compare2}
Suppose that $V_\bullet$ is birational. Then we have
$$
\widetilde{\xi}(V_\bullet; x) \geq \eps(V_\bullet; x).
$$
\end{lemma}

\begin{proof}
It is enough to prove that
$$
\widetilde{\xi}:= \widetilde{\xi}_{\widetilde{Y}_\bullet}(V_\bullet;x) \geq \eps(V_\bullet; x) =: \eps
$$
for an infinitesimal admissible flag $\widetilde{Y}_\bullet$ over $x$.
For any integer $m \geq 1$, let $s_m:=s(V_m;x)$ so that the map
$$
V_m \longrightarrow H^0(\mathcal{O}_X(\lfloor mD \rfloor) \otimes \mathcal{O}_X/\mathfrak{m}_x^{s_m+1})
$$
is surjective. Recall that $\Gamma(V_\bullet)_m$ is the image of $\nu_{Y_\bullet} \colon (V_m \setminus \{ 0 \} ) \to \Z_{\geq 0}^n$. By \cite[Lemma 1.4]{LM} and Lemma \ref{lemma:upper bound2} $(1)$, we have
$$
h^0(\mathcal{O}_X(\lfloor mD \rfloor) \otimes \mathcal{O}_X/\mathfrak{m}_x^{s_m+1}) = \# \big( \Gamma(V_\bullet)_m \cap \big(m\widetilde{\okbd}_{\widetilde{Y}_\bullet}(V_\bullet) \big)_{x_1 < s_m + 1}\big) \leq \# \big(\Z_{\geq 0}^n \cap \widetilde{\blacktriangle}^n_{s_m} \big).
$$
However, the both end sides are the same. Thus $\widetilde{\blacktriangle}^n_{s_m} \subseteq m\widetilde{\okbd}_{\widetilde{Y}_\bullet}(V_\bullet)$, so 
$$
\widetilde{\blacktriangle}^n_{\frac{s_m}{m}} \subseteq \widetilde{\okbd}_{\widetilde{Y}_\bullet}(V_\bullet).
$$
Since $\displaystyle \eps = \limsup_{m \to \infty} \frac{s_m}{m}$, it follows that $\widetilde{\blacktriangle}^n_{\eps} \subseteq \widetilde{\okbd}_{\widetilde{Y}_\bullet}(V_\bullet)$. This implies that $\widetilde{\xi} \geq \eps$. 
\end{proof}

\subsection{The case when $V_\bullet$ is complete}

In this subsection, we prove Theorem \ref{main1} and Theorem \ref{main2}  $(1)$, $(2)$, $(3)$.
Throughout the subsection, we assume that $V_\bullet$ is a complete graded linear series of $D$. By Lemma \ref{lemma:upper bound2} $(2)$, 
$$
\widetilde{\xi}:=\widetilde{\xi}(V_\bullet;x)=\widetilde{\xi}_{\widetilde{Y}_\bullet}(V_\bullet;x)
$$
for any infinitesimal admissible flag $\widetilde{Y}_\bullet$ over $x$.
To prove Theorem \ref{main2}, by Lemmas \ref{lemma:compare} and \ref{lemma:compare2}, we only need to show that 
$$
\eps:=\eps(V_\bullet;x) \geq \widetilde{\xi}.
$$

\begin{proof}[Proof of Theorem \ref{main2} (1)]
We assume that $D$ is nef and big.
To derive a contradiction, suppose that $\eps < \widetilde{\xi}$. Take any number $k$ with $\eps < k < \widetilde{\xi}$. By Lemma \ref{lemma:curve}, 
$$
\mathbf{0} \not\in \okbd_{\widetilde{Y}_\bullet}(\pi^*D-k E) 
$$
for some infinitesimal admissible flag $\widetilde{Y}_\bullet$ over $x$.
However, $\widetilde{\blacktriangle}_{\widetilde{\xi}}^n \subseteq \widetilde{\okbd}_{\widetilde{Y}_\bullet}(D)$. Thus
$$
\mathbf{0} \in \widetilde{\okbd}_{\widetilde{Y}_\bullet}(D)_{x_1 \geq k} + (-k, \underbrace{0, \ldots, 0}_{n-1~\text{times}})=\okbd_{\widetilde{Y}_\bullet}(\pi^*D-k E),
$$ 
which is a contradiction. 
Hence $\eps \geq \widetilde{\xi}$, so we finish the proof.
\end{proof}

The following is a more comprehensive version of Theorem \ref{main1}.

\begin{theorem}\label{thm-main1}
Let $X$ be a smooth projective variety of dimension $n$, and $D$ be a big $\R$-divisor on $X$. Then the following are equivalent:
\begin{enumerate}[wide, labelindent=0pt]
 \item[$(1)$] $D$ is ample.
 \item[$(2)$] For every admissible flag $Y_\bullet$ on $X$, the Okounkov body $\okbd_{Y_\bullet}(D)$ contains a nontrivial standard simplex in $\R_{\geq 0}^n$.
 \item[$(3)$] For every point $x \in X$, there is an admissible flag $Y_\bullet$ centered at $x$ such that $\okbd_{Y_\bullet}(D)$ contains a nontrivial standard simplex in $\R_{\geq 0}^n$.
 \item[$(4)$] For every infinitesimal admissible flag $\widetilde{Y}_\bullet$ over $X$, the infinitesimal Okounkov body $ \widetilde{\okbd}_{\widetilde{Y}_\bullet}(D)$ contains a nontrivial inverted standard simplex in $\R_{\geq 0}^n$.
 \item[$(5)$] For every point $x \in X$, there is an infinitesimal admissible flag $\widetilde{Y}_\bullet$ over $x$ such that $\widetilde{\okbd}_{\widetilde{Y}_\bullet}(D)$ contains a nontrivial inverted standard simplex in $\R_{\geq 0}^n$.
\end{enumerate}
\end{theorem}

\begin{proof}
\noindent$(1) \Rightarrow (2)$: It can be shown by a standard argument (see e.g., \cite[Lemma 6.1]{CHPW}), so we skip the proof.

\noindent$(2) \Rightarrow (3)$: It is trivial.

\noindent$(3) \Rightarrow (4)$: If $(3)$ holds, then Lemma \ref{lemma:compare} implies that
$0 < \xi(D;x) \leq \widetilde{\xi}$.
Thus $(4)$ follows.

\noindent$(4) \Rightarrow (5)$: It is trivial.

\noindent$(5) \Rightarrow (1)$: Assume that $(5)$ holds. For any point $x \in X$, there is an infinitesimal admissible flag $\widetilde{Y}_\bullet$ over $x$ centered at $x' \in \widetilde{X}$ such that $\mathbf{0} \in \widetilde{\Delta}_{\widetilde{Y}_\bullet}(D)$. Lemma \ref{lemma:origin} says that $x' \not\in \bm(\pi^*D)$, and hence, $E \not\subseteq \bm(\pi^*D)$. Thus $x \not\in \bm(D)$. Since $x$ is an arbitrary point, it follows that $\bm(D) = \emptyset$. Thus $D$ is nef and big. 
Now, by Theorem \ref{main2} $(1)$, $\eps(D;x) >0$ for all $x \in X$.
By Lemma \ref{lem:seshample}, $D$ is ample, so $(1)$ holds.
\end{proof}

\begin{lemma}\label{lem:continuity}
Let $\{ A_i \}$ be a sequence of ample divisors on $X$ such that  $\displaystyle \lim_{i \to \infty} A_i = 0$. 
If $\widetilde{\xi}(D+A_i; x) = \eps(||D+A_i||;x)$ for all $i$, then $\widetilde{\xi}(D;x)=\eps(||D||;x)$ in the following cases:
 \begin{enumerate}[wide]
 \item[$(1)$] $(\cha(\Bbbk) = 0)$ $D$ is big.
 \item[$(2)$] $(\cha(\Bbbk) > 0)$ $D$ is big and $x \not\in \bp(D)$.
 \end{enumerate}
\end{lemma}

\begin{proof}
For an infinitesimal admissible flag $\widetilde{Y}_\bullet$, we have
$$
\widetilde{\okbd}_{\widetilde{Y}_\bullet}(D) \subseteq \widetilde{\okbd}_{\widetilde{Y}_\bullet}(D+A_i)~~\text{ and }~~\widetilde{\okbd}_{\widetilde{Y}_\bullet}(D)=\bigcap_{i \geq 1} \widetilde{\okbd}_{\widetilde{Y}_\bullet}(D+A_i).
$$
This type of statement first appeared in \cite[Lemma 8]{AKL}. We refer to the proof of \cite[Lemma 8]{AKL}. Consequently, we obtain
$$
\widetilde{\xi}(D;x) = \lim_{i \to \infty} \widetilde{\xi}(D+A_i; x).
$$
On the other hand, notice that $x \not\in \bp(D+A_i)$ for all $i$. By the continuity of $\eps(||-||;x)$ at $D$ (\cite[Proposition 6.3]{ELMNP2}) and \cite[Proposition 7.1.2]{Mur}), we get
$$
\eps(||D||;x)=\lim_{i \to \infty} \eps(||D+A_i||;x).
$$
Thus the lemma follows.
\end{proof}

\begin{proof}[Proof of Theorem \ref{main2} (2)]
We assume that $D$ is big and $x \not\in \bp(D)$.
We can take a sequence $\{ A_i \}$ of ample divisors on $X$ such that $D+A_i$ is a $\Q$-divisor and $\displaystyle \lim_{i \to \infty} A_i = 0$. By Lemma \ref{lem:continuity}, it is sufficient to show that $ \widetilde{\xi}(D+ A_i ;x) = \eps(||D+A_i||;x)$. Hence we may assume that $D$ is a $\Q$-divisor.
For a sufficiently large and divisible integer $m \geq 1$, let $f_m \colon X_m \to X$ be a birational morphism with the decomposition $f_m^*|mD| = M_m + F_m$ and $M_m' = \frac{1}{m}M_m, F_m' = \frac{1}{m}F_m$ as in Subsection \ref{subsec-nota}. We may assume that $f_m$ is isomorphic over a neighborhood of $x$ and $f_m^{-1}(x) \not\subseteq \Supp(F_m)$.

Suppose that $\eps < \widetilde{\xi}$. Take a rational number $k$ with $\eps < k < \widetilde{\xi}$. By Lemma \ref{lemma:curve}, there is a point $x' \in E$ such that 
\begin{equation}\label{eq-x'SB}
x' \in \bm(\pi_m^*M_m' - kE) \subseteq \text{SB}(\pi_m^*M_m' - kE).
\end{equation}
We claim that 
\begin{equation}\tag{$\star$}\label{claimstar}
x' \not\in \text{SB}(\pi^*D - kE).
\end{equation}
Granting the claim for now, we derive a contradiction to conclude that $\eps = \widetilde{\xi}$. 
By the claim,   
$$
x' \not\in \text{Bs}(m(\pi^*D-kE))~\text{ for a sufficiently large and divisible integer $m>0$.}
$$
We have
$$
m(\widetilde{f}_m^*\pi^*D - kE) = \pi_m^*(M_m + F_m) - mkE = (\pi_m^*M_m - mkE) + \pi_m^*F_m.
$$
Since $\widetilde{X}_m$ is normal and $F_m$ is the fixed part of $f_m^*|mD|$ with $f_m^{-1}(x) \not\in \Supp(F_m)$, we may identify
$$
|m(\pi^*D-kE)| =  |\pi_m^*M_m - mkE| + \pi_m^*F_m.
$$
Thus $x' \not\in \text{Bs}(\pi_m^*M_m - mkE)$, which is a contradiction to (\ref{eq-x'SB}).

It only remains to show the claim (\ref{claimstar}). There is an ample divisor $A$ on $X$ such that $F =D-A$ is an effective divisor and $x \not\in \Supp(F)$. Take a sufficiently small number $\delta > 0$ such that $A':=\pi^*A - \delta E$ is ample and a sufficiently small number $\epsilon > 0$ such that 
$$
\frac{k-\epsilon \delta}{1- \epsilon} < \widetilde{\xi}~~\text{ and }~~\bp(\pi^*D-kE)=\bm(\pi^*D-kE - \epsilon A').
$$
We only have to show $x' \not\in \bm(\pi^*D-kE - \epsilon A')$.
Now, notice that
$$
\mathbf{0} \in (1-\epsilon) \okbd_{\widetilde{Y}_\bullet}\left(\pi^*D -\frac{k-\epsilon \delta}{1-\epsilon}E \right) = \okbd_{\widetilde{Y}_\bullet}((1- \epsilon)\pi^*D - (k-\epsilon \delta)E)
$$ 
for an infinitesimal admissible flag $\widetilde{Y}_\bullet$ over $x$ centered at $x'$. By Lemma \ref{lemma:origin}, 
$$
x' \not\in \bm((1- \epsilon)\pi^*D - (k-\epsilon \delta)E),
$$
and hence,
$$
x' \not\in \bm((1- \epsilon)\pi^*D - (k-\epsilon \delta)E + \epsilon \pi^*F).
$$
But we have
$$
(1-\epsilon)\pi^*D - (k-\epsilon \delta)E + \epsilon \pi^*F=(1-\epsilon)\pi^*D - (k-\epsilon \delta)E
+\epsilon (\pi^*D - A'-\delta E)= \pi^*D-kE - \epsilon A'.
$$
This finishes the proof.
\end{proof}

\begin{proof}[Proof of Theorem \ref{main2} (3)]
We assume that ${\rm char}(\Bbbk)=0$ and $D$ is big.
By Theorem \ref{main2} $(2)$, we only have to consider the case that $x \in \bp(D)$. In this case, we know that 
$$
\eps=\eps(||D||;x)=0.
$$ 
Thus it suffices to show that $\widetilde{\xi}=0$. 
If $x \in \bm(D)$, then $E \subseteq \bm(\pi^*D)$ so that $\widetilde{\xi}=0$ by Lemma \ref{lemma:origin}.
Next, assume that $x \in \bp(D) \setminus \bm(D)$. Take an ample divisor $A$ on $X$. 
For any number $\epsilon > 0$, we have $x \not\in \bp(D+\epsilon A)$ , so Theorem \ref{main2} $(2)$ implies that 
$$
\widetilde{\xi}(D+\epsilon A; x) = \eps(||D+\epsilon A||;x ).
$$
By Lemma \ref{lem:continuity}, $\widetilde{\xi} = \eps=0$. We finish the proof.
\end{proof}

As a consequence of Theorem \ref{main2} $(3)$ and \cite[Proposition 6.8]{ELMNP2}, we can obtain \cite[Proposition 4.10]{KL1}, which describes the jet separation of $K_X+D$. For another application, we recover \cite[Theorem C]{CHPW} and \cite[Theorem 4.1]{KL1} as follows.

\begin{corollary}[${\rm char}(\Bbbk)=0$] \label{theorem:0 equivalence}
The following are equivalent:
\begin{enumerate}[wide, labelindent=0pt]
 \item[$(1)$] $x \not\in \bp(D)$.
 \item[$(2)$] The Okounkov body $\okbd_{Y_\bullet}(D)$ contains a nontrivial standard simplex in $\R_{\geq 0}^n$ for every admissible flag $Y_\bullet$ on $X$ centered at $x$.
 \item[$(3)$] The Okounkov body $\okbd_{Y_\bullet}(D)$ contains a nontrivial standard simplex in $\R_{\geq 0}^n$  for some admissible flag $Y_\bullet$ on $X$ centered at $x$.
 \item[$(4)$] The infinitesimal Okounkov body $ \widetilde{\okbd}_{\widetilde{Y}_\bullet}(D)$ contains a nontrivial inverted standard simplex in $\R_{\geq 0}^n$  for every infinitesimal admissible flag $\widetilde{Y}_\bullet$ over $x$.
 \item[$(5)$] The infinitesimal Okounkov body $ \widetilde{\okbd}_{\widetilde{Y}_\bullet}(D)$ contains a nontrivial inverted standard simplex in $\R_{\geq 0}^n$  for some infinitesimal admissible flag $\widetilde{Y}_\bullet$ over $x$.
\end{enumerate}
\end{corollary}

\begin{proof}
\noindent$(1) \Rightarrow (2)$: It can be shown by a standard argument (see e.g., \cite[Theorem C]{CHPW}), so we skip the proof.

\noindent$(2) \Rightarrow (3) \Rightarrow (4) \Rightarrow (5)$: The proofs are identical to those of Theorem \ref{main1}.

\noindent$(5) \Rightarrow (1)$: Assume that $(5)$ holds. By Theorem \ref{main2} $(3)$, we have $\eps(||D||;x)>0$. Thus $x \not\in \bp(D)$, i.e., $(1)$ holds.
\end{proof}

\begin{remark}\label{remark:char0assump}
The characteristic zero assumption in Theorem \ref{main2} $(3)$ and Corollary \ref{theorem:0 equivalence} is used only when \cite[Theorem 6.2]{ELMNP2} is applied.
When $x \in \bp(D) \setminus \bm(D)$, Theorem \ref{main2} $(3)$ is equivalent to that $\lim_{\epsilon \to 0}\eps (||D+\epsilon A||;x) = 0$. To remove the characteristic zero assumption, by \cite[Proposition 7.1.2 (1)]{Mur}, it is enough to extend the following deep result in \cite{ELMNP2} to positive characteristics:
$$
\text{if $V$ is an irreducible component of $\bp(D)$, then}~\lim_{\epsilon \to 0} \vol_{X|V}(D+\epsilon A) = 0.
$$ 
\end{remark}

\subsection{The case when $V_\bullet$ is birational}

In this subsection, we prove Theorem \ref{main2} $(4)$. 

\begin{proof}[Proof of Theorem \ref{main2} (4)]
We assume that $V_\bullet$ is birational, $x \in X$ is very general, and ${\rm char}(\Bbbk)=0$.
By Lemmas \ref{lemma:upper bound2} $(2)$,  \ref{lemma:compare}, and \ref{lemma:compare2}, we only need to show that 
$$
\eps=\eps(V_\bullet;x) \geq \widetilde{\xi}_{\widetilde{Y}_\bullet}(V_\bullet;x)=\widetilde{\xi}
$$
for an infinitesimal admissible flag $\widetilde{Y_\bullet}$ over $x$.
We use the notations in Subsection \ref{subsec-nota}. 
Since $x$ is a very general point, we may assume that a birational morphism $f_m \colon X_m \to X$ is isomorphic over a neighborhood of $x$ and $f_m^{-1}(x) \not\subseteq \Supp(F_m)$. 
We may identify $f_m^{-1}(x)$ with $x$ and $E_m$ with $E$. We can regard $\widetilde{Y}_\bullet$ as an infinitesimal admissible flag over $x$ on $\widetilde{X}_m$. Let $V_\bullet^m$ be the graded linear series such that $V_i^m$ is the image of the map $S^iV_m \to V_{im}$. Then we have
\begin{equation}\label{eq:okbdvm}
\frac{1}{m} \widetilde{\okbd}_{\widetilde{Y}_\bullet}(V_\bullet^m) \subseteq \frac{1}{m} \widetilde{\okbd}_{\widetilde{Y}_\bullet}(M_m) = \widetilde{\okbd}_{\widetilde{Y}_\bullet}(M_m').
\end{equation}

To derive a contradiction, suppose that $\eps <\widetilde{\xi}$.
Recall from \cite[Lemma 3.10]{I} that
$$
\eps =\eps(V_\bullet;x)  = \sup_{m >0} \eps(M_m'; x).
$$ 
For any $k$ with $\eps(M_m' ; x) < k < \widetilde{\xi}$, by Lemma \ref{lemma:curve}, there is an irreducible curve $C_k$ on $X$ passing through $x$ such that 
$$
\overline{C}_k \subseteq \bm(\pi_m^*M_m'-kE) \subseteq \text{SB}(\pi_m^*M_m' - kE),
$$ 
where $\overline{C}_k$ is the strict transform of $C_k$ by $f_m \circ \pi_m = \pi \circ \widetilde{f}_m$.
Let
$$
\alpha(C_k):=\inf \{\beta \in \Q \mid \overline{C}_k \subseteq \text{SB}(\pi_m^*M_m'-\beta E) \}.
$$
Let $\widetilde{Y}_\bullet$ be an infinitesimal admissible flag over $x$ centered at a point $x' \in \overline{C}_k \cap E$.
By \cite[Proposition 2.3]{EKL}, \cite[Lemma 1.3]{N}, for any $\beta$ with $\alpha(C_k) < \beta < \mu(M_m';x)$,  we have
$$
\ord_{x'} (||\pi_m^*M_m'-\beta E||) \geq \ord_{\overline{C}_k} (||\pi_m^*M_m' - \beta E||) \geq \beta - \alpha(C_k).
$$
For any point 
$$
(\nu_1, \ldots, \nu_{n-1}) \in \widetilde{\okbd}_{\widetilde{Y}_\bullet}(M_m')_{x_1=\beta} = \okbd_{\widetilde{Y}_\bullet}(\pi_m^*M_m' - \beta E)_{x_1=0},
$$ 
by considering Lemma \ref{lemma:upper bound} $(2)$ and (\ref{eq-okbd_mult}) in Lemma \ref{lemma:origin}, we have
$$
\nu_1 + \cdots + \nu_{n-1} \geq \beta - \alpha(C_k).
$$
This implies that 
$$
\text{interior } \big(  \widetilde{\okbd}_{\widetilde{Y}_\bullet}(M_m')_{x_1=\beta}   \big) \cap \text{interior } \big(\blacktriangle_{\beta-\alpha(C_k)}^{n-1} \big) = \emptyset~~\text{ in $\R_{\geq 0}^{n-1}$}.
$$
Note that $\alpha(C_k) \leq k$ so that 
$$
\lim_{k \to \eps(M_m' ;x)} \alpha(C_k) \leq \eps(M_m' ;x).
$$
By considering (\ref{eq:okbdvm}), we see that
$$
\text{interior} \left(\frac{1}{m}\widetilde{\okbd}_{\widetilde{Y}_{\bullet}}(V_\bullet^m)\right) \cap \text{interior }  \big((\eps(M_m';x), \underbrace{0, \ldots, 0}_{n-1~\text{times}}) +\widetilde{\blacktriangle}_{\widetilde{\xi} - \eps(M_m'; x)}^{n} \big) = \emptyset~~\text{in $\R_{\geq 0}^n$}.
$$
Since we have
$$
\frac{1}{m}\widetilde{\okbd}_{\widetilde{Y}_{\bullet}}(V_\bullet^m),~~(\eps(M_m';x), \underbrace{0, \ldots, 0}_{n-1~\text{times}}) +\widetilde{\blacktriangle}_{\widetilde{\xi} - \eps(M_m'; x)}^{n} \subseteq \widetilde{\okbd}_{\widetilde{Y}_\bullet}(V_\bullet),
$$
it follows from \cite[Theorem 2.13]{LM} that
$$
\vol_X(V_\bullet) -  \frac{1}{m^n}\vol_{X}(V_\bullet^m) \geq  (\widetilde{\xi}-\eps(M_m';x)))^n \geq (\widetilde{\xi}-\eps)^n.
$$
However, \cite[Theorem D and Theorem 2.13]{LM} says that $\vol_X(V_\bullet) - \frac{1}{m^n} \vol_{X}(V_\bullet^m)$ is arbitrarily small for a sufficiently large integer $m \gg 0$, so we get a contradiction. Therefore, $\eps = \widetilde{\xi}$, and we complete the proof.
\end{proof}

\begin{remark}\label{remark:gls}
We assume ${\rm char}(\Bbbk)=0$.
One can easily check that if $V_\bullet$ is finitely generated, then Theorem \ref{main2} $(4)$ holds for every point $x \in X$. However, in general, Theorem \ref{main2} $(4)$ may not hold when the point $x$ is not general. For example, we fix a point $x \in \P^2$, and consider a graded linear series $V_\bullet$  associated to $\mathcal{O}_{\P^2}(1)$ given by
$$
V_m:=\{s \in H^{0}(\mathbb{P}^{2}, \mathcal{O}_{\mathbb{P}^{2}}(m)) \text{ }|\text{ {\rm ord}}_{x}(s) \geq 1\}
$$
for any $m \geq 1$. Evidently, $V_\bullet$ is birational. For any infinitesimal admissible flag $\widetilde{Y}_\bullet$ over $x$ or any admissible flag $Y_\bullet$ centered at $x$, we have
$$
\widetilde{\okbd}_{\widetilde{Y}_\bullet}(V_\bullet)=\widetilde{\okbd}_{\widetilde{Y}_\bullet}(\mathcal{O}_{\P^2}(1))~~\text{ and }~~\okbd_{Y_\bullet}(V_\bullet)=\okbd_{Y_\bullet}(\mathcal{O}_{\P^2}(1)),
$$
so we obtain $\widetilde{\xi}(V_\bullet;x)=\xi(V_\bullet;x)=1$. However, we have $\eps(V_\bullet; x)=0$.
\end{remark}

\section{Integrated volume functions}\label{sec_filokbd}

This section is devoted to the study of integrated volume functions; in particular, we prove Theorem \ref{main3}. 
For a given admissible flag $Y_\bullet$ on $X$ centered at $x$ and a graded linear series $V_\bullet$, we define the filtered Okounkov body $\widehat{\okbd}_{Y_\bullet}(V_\bullet, \mathcal{F}_x) \subseteq \R^{n+1}_{\geq 0}$ in Section \ref{sec_prelim}.

\begin{definition}
The \emph{integrated volume function} of $(V_\bullet, \mathcal{F}_x)$ at $x$ is defined by
$$
\widehat{\varphi}_x(V_\bullet, \mathcal{F}_x, t):=\int_{u=0}^{t} {\rm vol}_{\mathbb{R}^{n}}(\widehat{\Delta}_{Y_{\bullet}}(V_{\bullet}, \mathcal{F}_x)_{x_{n+1}=u})du.
$$
\end{definition}

\begin{remark}
We can easily check that the function
$$
\widehat{\varphi}_x(V_\bullet, \mathcal{F}_x, -) \colon \R_{\geq 0} \to \R_{\geq 0}, ~~t \mapsto \widehat{\varphi}_x(V_\bullet, \mathcal{F}_x, t)
$$ 
is nondecreasing and continuous.
When $V_\bullet$ is a complete graded linear series of $D$, we put $\widehat{\varphi}_x(D,t):=\widehat{\varphi}_x(V_\bullet, \mathcal{F}_x, t)$. Clearly, if $D \equiv D'$, then $\widehat{\varphi}_x(D,t)=\widehat{\varphi}_x(D',t)$. Then the function
$$
\widehat{\varphi}_{x}: {\rm Big}(X) \times \mathbb{R}_{\geq 0} \rightarrow \mathbb{R}_{\geq 0}, ~~ (D,t) \mapsto \widehat{\varphi}_x(D,t)
$$
is continuous on the whole domain.
\end{remark}

\begin{example}
Let $X=\P^1 \times \P^1$, and $V_\bullet$ be the complete graded linear series associated to $D=\mathcal{O}_{\P^1}(1) \boxtimes \mathcal{O}_{\P^1}(1)$. Note that $\mu(V_\bullet;x)=2$. For any admissible flag $Y_\bullet$ on $X$ centered at $x$, we have
$$
\vol_{\R^2}\left( \widehat{\Delta}_{Y_{\bullet}}(V_{\bullet}, \mathcal{F}_x)_{x_{n+1}=t} \right)= \vol_{\R^2}\left( \okbd_{Y_\bullet}(V_\bullet^{(t)})\right) =
\begin{cases}
-\frac{1}{2}t^2+1 &\text{if $0 \leq t \leq 1$,}\\
\frac{1}{2}t^2-2t+2 &\text{if $1 \leq t \leq 2$.}
\end{cases}
$$
It then follows that
$$
\widehat{\varphi}_x(V_\bullet, \mathcal{F}_x,t) =
 \begin{cases}
-\frac{1}{6}t^3+t &\text{if $0 \leq t \leq 1$,}\\
\frac{1}{6}t^3 - t^2+2t-\frac{1}{3}&\text{if $1 \leq t \leq 2$.}
\end{cases}
$$
See \cite[Example in Section 4]{MR} for the case that $D=\mathcal{O}_{\P^1}(d_1) \boxtimes \mathcal{O}_{\P^1}(d_2)$ with $d_1, d_2 \geq 1$.
\end{example}

We are now in a position to extract several important invariants of graded linear series from the integrated volume function. In particular, Proposition \ref{prop:basicfilvol} $(6)$ together with Theorem \ref{main2} immediately implies Theorem \ref{main3}.

\begin{proposition}\label{prop:basicfilvol}
We have the following:
\begin{enumerate}[wide, labelindent=0pt]
 \item[$(1)$] $\widehat{\varphi}_x(V_\bullet, \mathcal{F}_x, t)=  \widehat{\varphi}_x(V_\bullet, \mathcal{F}_x, \mu(V_\bullet;x))=\vol_{\R^{n+1}} \left( \widehat{\okbd}_{Y_\bullet}(V_\bullet, \mathcal{F}_x) \right)$ for any $t \geq \mu(V_\bullet;x)$.
 \item[$(2)$] $\widehat{\varphi}_x'(V_\bullet, \mathcal{F}_x, t)=\frac{d \widehat{\varphi}_x(V_\bullet, \mathcal{F}_x, t)}{dt} = \vol_{\R^n}(\okbd_{Y_\bullet}(V_\bullet^{(t)}))$ for all $t \geq 0$.
 \item[$(3)$] $\vol_X(V_\bullet^{(t)}) = n! \cdot \widehat{\varphi}_{x}'(V_\bullet, \mathcal{F}_x,t)$ for all $t \geq 0$.
 \item[$(4)$] $\widehat{\varphi}_{x}'(V_\bullet, \mathcal{F}_x, 0)-\widehat{\varphi}_{x}'(V_\bullet, \mathcal{F}_x, t) \leq \frac{t^n}{n!}$ for all $t \geq 0$.
 \item[$(5)$] $\mu(V_\bullet; x)=\inf \{t \geq 0 \mid \widehat{\varphi}_{x}'(V_\bullet, \mathcal{F}_x,t)=0\}$.
 \item[$(6)$] $\widetilde{\xi}(V_\bullet;x) = \inf \left \{ t \geq 0 \suchthat \widehat{\varphi}_{x}'(V_\bullet, \mathcal{F}_x, 0)-\widehat{\varphi}_{x}'(V_\bullet, \mathcal{F}_x, t)<\frac{t^n}{n!} \right \}$.
\end{enumerate}
\end{proposition}

\begin{proof} $(1)$ and $(2)$ are clear by the definition. Then $(3)$ follows from $(2)$ and \cite[Theorem 2.13]{LM}. Now, fix an infinitesimal admissible flag $\widetilde{Y}_\bullet$ over $x$. Then we have
$$
\widehat{\varphi}_x'(V_\bullet, \mathcal{F}_x, t)=\vol_{\mathbb{R}^{n}}(\Delta_{Y_{\bullet}}(V_{\bullet}^{(t)}))=\vol_{\R^n}(\widetilde{\okbd}_{\widetilde{Y}_\bullet}(V_\bullet)_{x_1 \geq t} ).
$$
By Lemma \ref{lemma:upper bound2} $(1)$ and \cite[Theorem 2.13]{LM}, we obtain $(4)$ (see \cite[Lemma 4.1]{MR} for an alternative proof of $(4)$ when $V_\bullet$ is a complete graded linear series). Observe that
$$
\mu(V_\bullet;x)=\sup\{ \nu_1 \mid \mathbf{x}=(\nu_1, \ldots, \nu_n) \in \widetilde{\okbd}_{\widetilde{Y}_\bullet}(V_\bullet) \}.
$$ 
This implies $(5)$. Note that
$$
\widehat{\varphi}_{x}'(V_\bullet, \mathcal{F}_x, 0)-\widehat{\varphi}_{x}'(V_\bullet, \mathcal{F}_x, t) = \vol_{\R^n}(\widetilde{\okbd}_{\widetilde{Y}_\bullet}(V_\bullet)_{0 \leq x_1 \leq t}).
$$
Then $(6)$ follows from Lemma \ref{lemma:upper bound2} $(1)$ and $(2)$.
\end{proof}

Recall from \cite[Definition 1.2]{BC} that the \emph{jumping numbers} of $(V_m, \mathcal{F}_x)$ are defined as
$$
e_\ell(V_m)=e_\ell(V_m, \mathcal{F}_x):=\sup \{ t \in \R \mid \dim \mathcal{F}_x^t V_m \geq \ell \}~~ \text{for $\ell=1, \ldots, \dim V_m=v_m$}.
$$
We have 
$$
0 \leq e_{v_m}(V_m) \leq \cdots \leq e_1(V_m).
$$ 
The \emph{positive mass} of $(V_m, \mathcal{F}_x)$ is defined as
$$
\mass_+(V_m)=\mass_+(V_m, \mathcal{F}_x) := \sum_{e_j(V_m) > 0} e_j(V_m) = \sum_{1 \leq j \leq v_m} e_j(V_m).
$$

\begin{definition}
\begin{enumerate}[wide, labelindent=0pt]
\item Let 
$$
S(V_m)=S(V_m, \mathcal{F}_x):=\{ e_{v_m}(V_m), \ldots, e_1(V_m) \}~~\text{ and }~~N(V_m)=N(V_m, \mathcal{F}_x):=|S(V_m)|.
$$
\item We define the \emph{effective jumping numbers} of $(V_m, \mathcal{F}_x)$ as
$$
\alpha_j(V_m)=\alpha_j(V_m, \mathcal{F}_x):=\text{the $j$-th largest element in $S(V_m)$ for $j=1, \ldots, N(V_m)$}.
$$
For convention, we put $\alpha_{N(V_m)+1}:=0$. We have
$$
0=\alpha_{N(V_m)+1}(V_m) \leq \alpha_{N(V_m)}(V_m) < \alpha_{N(V_m)-1}(V_m) < \cdots < \alpha_1(V_m).
$$
\item Let 
$$
\begin{array}{l}
\beta_j(V_m)=\beta_j(V_m, \mathcal{F}_x):=\max\{ \ell \in [1, v_m] \mid e_\ell(V_m) = \alpha_j(V_m) \}~\text{ for $j=1, \ldots, N(V_m)$},\\
j_t(V_m)=j_t(V_m, \mathcal{F}_x):=\inf\{ j \in [1, N(V_m)+1] \mid \alpha_j(V_m) \leq t \}~\text{ for $t \geq 0$}.
\end{array}
$$
\item For $t \geq 0$, the \emph{bounded mass function} of $(V_m, \mathcal{F}_x)$ is defined as
$$
\mass_+(V_m, t)=\mass_+(V_m, \mathcal{F}_x, t)
:=\beta_{j_t-1}(t-\alpha_{j_t})+\sum_{j=j_t}^{N(V_m)} \beta_j (\alpha_j - \alpha_{j+1}),
$$
where $\alpha_{j}=\alpha_{j}(V_{m})$, $\beta_{j}=\beta_{j}(V_{m})$, and $j_{t}=j_{t}(V_{m})$. When $j_t(V_m)=N(V_m)+1$, we put 
$$
\sum_{j=j_t(V_m)}^{N(V_m)} \beta_j (V_m)(\alpha_j(V_m) - \alpha_{j+1}(V_m)):=0.
$$
\end{enumerate}
\end{definition}

\begin{remark}
One can check that
$$
\mass_+(V_m)=\sum_{e_j(V_m) > 0} e_j(V_m)=\sum_{j=1}^{N(V_m)} \beta_j(V_m) (\alpha_j(V_m) - \alpha_{j+1}(V_m))=\mass_+(V_m, \infty).
$$
However, $\mass_+(V_m, t) \neq \sum_{0 < e_j(V_m) < t} e_j(V_m)$ in general.
\end{remark}

We now show that the integrated volume function can be expressed in terms of the bounded mass functions. In particular, we see that the integrated volume function $\widehat{\varphi}_{x}(V_\bullet, \mathcal{F}_x, t)$ only depends on the multiplicative filtration $\mathcal{F}_x$ on $V_\bullet$.

\begin{theorem}\label{theorem:mass}
For all $t \geq 0$, we have
$$
\widehat{\varphi}_{x}(V_\bullet, \mathcal{F}_x, t)=\lim_{m \rightarrow \infty} \frac{\mass_+ (V_{m}, \mathcal{F}_x, mt)}{m^{n+1}}.
$$
\end{theorem}

\begin{proof}
For any $t \geq 0$, we have
$$
\widehat{\varphi}_{x}(V_\bullet, \mathcal{F}_x, t)=\int_{0}^{t} {\rm vol}_{\mathbb{R}^{n}}(\Delta_{Y_{\bullet}}(V_{\bullet}^{(u)}))du=\int_{0}^{t} \lim_{m\rightarrow \infty}\frac{{\rm dim}\mathcal{F}_{x}^{mu}V_{m}}{m^{n}}du.
$$
By \cite[Theorem 2.13]{LM}, the sequence of functions $\left\{f_{m}(u):=\frac{{\rm dim}\mathcal{F}_{x}^{mu}V_{m}}{m^{n}} \right\}$ converges pointwise to the function $f(u):={\rm vol}_{\mathbb{R}^{n}}(\Delta_{Y_{\bullet}}(V_{\bullet}^{(u)}))$ as $m \to \infty$.
We have $|f_{m}(u)| \leq \frac{{\rm dim}V_{m}}{m^{n}}$. By \cite[Theorem 2.13 and Remark 2.14]{LM},
$$
\limsup_{m \rightarrow \infty}\frac{{\rm dim}V_{m}}{m^{n}}=\lim_{m \rightarrow \infty}\frac{{\rm dim}V_{m}}{m^{n}}=\frac{{\rm vol}_{X}(V_{\bullet})}{n!},
$$ 
so there exists a constant $C>0$ such that $
 \suchthat \frac{{\rm dim}V_{m}}{m^{n}}  \suchthat \leq C$ for all $m \gg 0$. It is enough to consider only a sufficiently large $m \gg 0$ for the claim, so we may assume that $|f_{m}(u)| \leq C$ for all $m$. The integration is taken over an area $(0,t)$ with a finite measure, and hence, we see that $|f_{m}(u)|$ is bounded by an integrable function. Now, by applying Lebesgue dominated convergence theorem and change of variable $v=mu$, we obtain
$$
\int_{0}^{t} \lim_{m\rightarrow \infty}\frac{{\rm dim}\mathcal{F}_{x}^{mu}V_{m}}{m^{n}}du
=\lim_{m\rightarrow \infty}\int_{0}^{t} \frac{{\rm dim}\mathcal{F}_{x}^{mu}V_{m}}{m^{n}}du=\lim_{m\rightarrow \infty}\frac{1}{m^{n+1}}\int_{0}^{mt} {\rm dim}\mathcal{F}_{x}^{v}V_{m}dv.
$$
By regarding ${\rm dim}\mathcal{F}_{x}^{v}V_{m}$ as a function of $v$, we can write
$$
{\rm dim}\mathcal{F}_{x}^{v}V_{m}=
\begin{cases}
\beta_{N(V_{m})}(V_{m}) &\text{if $v = 0$,} \\
\beta_{j}(V_{m}) &\text{if $v \in (\alpha_{j+1}(V_{m}), \alpha_{j}(V_{m})]$ for $j=N(V_{m}), \ldots, 1$,} \\
0 &\text{if $v \in (\alpha_{1}(V_{m}), \infty)$.}
\end{cases}
$$
Now, it is immediate to see that
$$
\int_{0}^{mt} {\rm dim}\mathcal{F}_{x}^{v}V_{m}dv=\beta_{j_{mt}-1}(mt-\alpha_{j_{mt}})+ \sum_{j=j_{mt}}^{N(V_m)} \beta_j (\alpha_j - \alpha_{j+1})=\mass_+(V_{m},\mathcal{F}_{x},mt),
$$
where $\alpha_{j}=\alpha_{j}(V_{m})$, $\beta_{j}=\beta_{j}(V_{m})$, and $j_{mt}=j_{mt}(V_{m})$, which gives the desired result.
\end{proof}

As a consequence of Theorem \ref{theorem:mass}, we recover \cite[Corollary 1.13]{BC} in our situation.

\begin{corollary}
We have
$$
\vol_{\R^{n+1}} \left( \widehat{\okbd}_{Y_\bullet}(V_\bullet, \mathcal{F}_x) \right) = \widehat{\varphi}_{x}(V_\bullet, \mathcal{F}_x, \infty) = \lim_{m \rightarrow \infty} \frac{\mass_+ (V_{m}, \mathcal{F}_x)}{m^{n+1}}.
$$
\end{corollary}

\begin{example}
Let $X=\P^2$ be any point, and $V_\bullet$ be the complete graded linear series associated to $\mathcal{O}_{\P^2}(1)$. Note that $\mu(V_\bullet; x)=1$.
For integers $m,s>0$, we have 
$$
\dim \mathcal{F}_x^s V_m = \frac{(m+2)(m+1)}{2}-\frac{s(s+1)}{2}.
$$
The jumping numbers $e_j=e_j(V_m, \mathcal{F}_x)$ are given by
$$
e_{\frac{(m+2)(m+1)}{2}-\frac{k(k+1)}{2}}=\cdots=e_{\frac{(m+2)(m+1)}{2}-\frac{k(k+3)}{2}}=k
$$ 
for all $k \geq 0$ with $\frac{(m+2)(m+1)}{2}-\frac{k(k+3)}{2} \geq 1$, so the effective jumping numbers are given by 
$$
\alpha_j(V_m, \mathcal{F}_x)=m+1-j~~\text{ for $j=1, \ldots, m+1=N(V_m, \mathcal{F}_x)$.}
$$ 
We then obtain 
$$
\beta_j(V_m, \mathcal{F}_x)=\frac{(m+2)(m+1)}{2}-\frac{(m+1-j)(m+2-j)}{2}~\text{ and }~j_t(V_m, \mathcal{F}_x)=m+1-\lfloor t \rfloor.
$$
For $0 \leq t \leq 1$, the bounded mass function is given by
\begin{tiny}
$$
\mass_+ (V_{m}, \mathcal{F}_x, mt)=\left(\frac{(m+1)(m+2)}{2}-\frac{(\lfloor mt \rfloor+1)(\lfloor mt \rfloor+2)}{2}\right) \left(mt-\lfloor mt \rfloor \right) +  \sum_{k=1}^{\lfloor mt \rfloor}\left(\frac{(m+1)(m+2)}{2}-\frac{k(k+1)}{2} \right).
$$
\end{tiny}\\[-8pt]
By Theorem \ref{theorem:mass}, we obtain
$$
\widehat{\varphi}_{x}(V_\bullet, \mathcal{F}_x, t)=\lim_{m \rightarrow \infty} \frac{\mass_+ (V_{m}, \mathcal{F}_x, mt)}{m^{3}}=-\frac{1}{6}t^{3}+\frac{1}{2}t ~~\text{ for $0\leq t \leq 1$.}
$$
\end{example}


\end{document}